\documentclass{amsart}
\usepackage{amsmath,amsfonts,amsthm}
\usepackage{amssymb,latexsym}
\usepackage{graphics}
\usepackage{rotating}
\usepackage[colorlinks]{hyperref}
\usepackage[all]{xypic}
\usepackage{amssymb}
\usepackage{xspace, enumerate}
\theoremstyle{plain}
\theoremstyle{definition}
\newtheorem{theorem}{Theorem}[section]
\newtheorem{lemma}[theorem]{Lemma}

\newtheorem{corollary}[theorem]{Corollary}

\newtheorem{definition}[theorem]{Definition}

\newtheorem{note}[theorem]{Note}

\newtheorem{convention}[theorem]{Convention}
\newtheorem{remark}[theorem]{Remark}

\theoremstyle{remark}
\numberwithin{equation}{section}
\newcommand{\SP}{\: \: \: \: \:}
\newcommand{\h}{\mathfrak{h}}

\title{Specification properties on uniform spaces}
\author[{\tiny  F. Ayatollah Zadeh Shirazi, Z. Nili Ahmadabadi, B. Taherkhani and
Kh. Tajbakhsh}]{Fatemah Ayatollah Zadeh Shirazi, Zahra Nili Ahmadabadi, \\  Bahman Taherkhani,
Khosro Tajbakhsh}


\begin{document}
\begin{abstract}
In the following text we introduce specification property (stroboscopical property) for 
dynamical systems on uniform space.
We focus on  two classes of dynamical systems: generalized shifts and 
dynamical systems with Alexandroff compactification of a discrete space as phase space.
We prove that for a discrete finite topological space $X$ with at least two elements,
a nonempty set $\Gamma$ and a self--map $\varphi:\Gamma\to\Gamma$ the generalized shift
dynamical system $(X^\Gamma,\sigma_\varphi)$:
\begin{itemize}
\item has (almost) weak specification property if and only if $\varphi:\Gamma\to\Gamma$
	does not have any periodic point,
\item has (uniform) stroboscopical property if and only if $\varphi:\Gamma\to\Gamma$
	is one-to-one.
\end{itemize}
\end{abstract}
\maketitle
\section{Introduction}
\noindent The main purpose
of this paper is to generalize tracing properties to uniform spaces and in particular
to compact Hausdorff spaces (one may find similar ideas in \cite{Y92, 2013}). 
\\
The action of a group $G$ on metric space $(Z,d)$ is sensitive if there exists
$\epsilon>0$ such that for all $x\in Z$ and any open neighbourhood $U$ of $x$ there exists 
$y\in U$ and $g\in G$ with $d(gx,gy)\geq\epsilon$. T. Ceccherini--Silberstein and M.  Coornaert \cite{2013}
indicate that the action of a group $G$ on a uniform space $(Y,{\mathcal F})$ is sensitive if there exists
$\alpha\in{\mathcal F}$ such that for all $x\in Y$ and any open neighbourhood $U$ of $x$ there exists 
$y\in U$ and $g\in G$ with $(gx,gy)\notin\alpha$. In this paper we introduce specification properties  for 
dynamical systems on uniform space .
\\
We realize our generalizations in uniform spaces through two classes of dynamical systems
{\it ``generalized shift dynamical systems''} and {\it ``dynamical systems with Alexandroff
compactification of a discrete space as phase space''} on various specification and
stroboscopical properties.
\section{Preliminaries}
\noindent Let's recall that for an arbitrary set $Y$ we call the collection $\mathcal F$ of subsets of $Y\times Y$ a {\it uniform
structure on $Y$} if :
\begin{itemize}
\item $\forall\alpha\in{\mathcal F}\:(\Delta_Y\subseteq\alpha)$;
\item $\forall\alpha,\beta\in{\mathcal F}\:(\alpha\cap\beta\in{\mathcal F})$;
\item $\forall\alpha\in{\mathcal F}\:\exists\beta\in{\mathcal F}\:
    (\beta\circ\beta^{-1}\subseteq\alpha)$;
\item $\forall\alpha\in{\mathcal F}\:
    \forall\beta\subseteq Y\times Y\:(\alpha\subseteq\beta\Rightarrow\beta\in{\mathcal F})$.
\end{itemize}
where for $\alpha,\beta\subseteq X\times X$ we have
$\alpha^{-1}=\{(y,x):(x,y)\in\alpha\}$ and $\alpha\circ\beta=\{(x,z):\exists y\:((x,y)\in\beta\wedge (y,z)\in\alpha)
\}$ also $\Delta_Y=\{(x,x):x\in Y\}$) \cite{dug}.\\
\\
Moreover, for all $\alpha\in\mathcal F$ and $x\in Y$ let $\alpha[x]=\{y\in Y:(x,y)\in\alpha\}$.
For  uniform structure $\mathcal F$ on
$Y$, the collection  $\{U\subseteq Y:\forall x\in Y\:\exists\alpha\in{\mathcal F}\:
(\alpha[x]\subseteq U)\}$ is a topology on $Y$, we call it the {\it uniform topology} induced from $\mathcal F$,
so accordingly we denote this topological space by uniform topological space $(Y,{\mathcal F})$
whenever we want to emphasize  $\mathcal F$.
We call the topological space $Z$ {\it uniformizable} if there exists a
uniform structure $\mathcal G$ in $Z$
such that uniform topology induced from $\mathcal G$ coincides with original topology on $Z$, and
in this case we call $\mathcal G$ {\it compatible uniform structure on $Z$}.
Every compact Hausdorff (and in particular, compact metric)
space is uniformizable and has a unique compatible uniform structure.
\\
Note to the fact that compact metric space$(Z,d)$ has a unique 
compatible uniform structure $\{D\subseteq Z\times Z:\exists\kappa>0\:
(\{(x,y)\in Z\times Z:
d(x,y)<\kappa\}\subseteq D)\}$.
\\
In this text by a dynamical system $(Y,h)$ we mean a topological space $Y$
and continuous map $h:Y\to Y$.
Similar to the definition of specification properties in metric spaces 
\cite{D90} we have the following
definition.
\begin{definition}\label{def10}
For a uniform topological
	space $(Y,\mathcal{F})$, we say that the dynamical system
	$((Y,\mathcal{F}),h)$ (or briefly $(Y,h)$) has {\it almost weak specification property}
	if for all $\alpha\in\mathcal F$
there exists a map $N_\alpha:{\mathbb Z}_+\to{\mathbb Z}_+$ 
with ${\displaystyle\lim_{n\to\infty}\dfrac{N_\alpha(n)}{n}}=0$ 
such that for all $y_1,\ldots,y_n\in  Y$ ($n\geq2$) and
$0\leq l_1\leq k_1<l_2\leq k_2<\ldots<l_n\leq k_n$ with 
$l_{i+1}-k_i\geq N_\alpha(k_{i+1}-l_{i+1})$ 
(for all $i=1,\ldots,n-1$) there exists $x\in Y$ 
such that 
\[(h^i(x),h^i(y_j))\in\alpha\: \forall j\in\{1,\ldots,n\}\forall i\in\{l_j,l_j+1,\ldots,k_j\}\:.\]
If  in the above definition we suppose in addition that $N_\alpha$ is a constant map
(resp. and $x$ is a periodic point of $h$),
then we say that $((Y,\mathcal{F}),h)$ has {\it weak specification property}
(resp. {\it specification property}).
\end{definition}
\noindent Let's recall that in the dynamical system $(Z,f)$, for $x\in Z$ and
strictly increasing  sequence $A=\{n_i\}_{i\geq0}$ in $\mathbb N$
 we set $\omega_f(A,x)=\{y\in Z:\{f^{n_i}(x)\}_{i\geq0}$ has a subsequence converging to $y\}$.
 Now for definition of various stroboscopical properties see \cite{LS03}.
\begin{definition}\label{def20}
 We say that $(Z,f)$ has stroboscopical property (resp. strongly stroboscopical property)
 if for all $z\in Z$ and every
 strictly increasing  sequence $A=\{n_i\}_{i\geq0}$ in $\mathbb N$,
 $\{x\in Z:z\in\omega_f(A,x)\}$ is a nonempty (resp. dense) subset of $Z$.
 Moreover if $(Z,\mathcal F)$ is a uniform space and for all 
 strictly increasing  sequence $A=\{n_i\}_{i\geq0}$ in $\mathbb N$ there exists
 a subsequence $B=\{k_i\}_{i\geq 0}$ of $A$ and $\varrho:Z\to Z$
 such that the sequence $\{f^{k_i}\circ\varrho\}_{i\geq0}$ converges
 uniformly to $id_Z$, then we say that $((Z,\mathcal F),f)$ has uniform stroboscopical property. 
  \end{definition}
\subsection*{Generalized shifts}
The one sided shift $\mathop{\{1,\ldots,k\}^{\mathbb N}\to\{1,\ldots,k\}^{\mathbb N}}\limits_{
(x_n)_{n\geq1}\mapsto(x_{n+1})_{n\geq1}}$
and two sided shift $\mathop{\{1,\ldots,k\}^{\mathbb Z}\to\{1,\ldots,k\}^{\mathbb Z}}\limits_{
(x_n)_{n\in{\mathbb Z}}\mapsto(x_{n+1})_{n\in{\mathbb Z}}}$
due to P. Walters \cite{walters} have been investigated very well from the ergodic theory and dynamical systems point of 
view. 
For nonempty sets $X$ and $\Gamma$ 
and arbitrary map $\varphi:\Gamma\to\Gamma$,
we call $\sigma_\varphi:X^\Gamma\to X^\Gamma$ with
$\sigma_\varphi((x_\alpha)_{\alpha\in\Gamma})=(x_{\varphi(\alpha)})_{\alpha\in\Gamma}$
(for $(x_\alpha)_{\alpha\in\Gamma}\in X^\Gamma$), a {\it generalized shift}.
Whenever $X$ is a topological space and 
$X^\Gamma$ is equiped with product (pointwise convergence) topology, it's evident that 
$\sigma_\varphi:X^\Gamma\to X^\Gamma$ is continuous and one may consider
dynamical system $(X^\Gamma,\sigma_\varphi)$.
Generalized shifts has been introduced for the first time in \cite{AHK}, moreover
their dynamical properties have been studied in several texts, like
\cite{ANT13} and \cite{gio}.
\begin{convention}
Suppose that $X$ is a finite discrete topological space 
with at least two elements, $\Gamma$ is a nonempty set, and  $\varphi:\Gamma\to\Gamma$
is an arbitrary map. For $H\subseteq \Gamma$ let:
\[\alpha_H:=\{((x_\alpha)_{\alpha\in\Gamma},(y_\alpha)_{\alpha\in\Gamma})\in 
X^\Gamma\times X^\Gamma:\forall\alpha\in H\:(x_\alpha=y_\alpha)\}\:,\]
so 
\begin{center}
${\mathcal U}:=\{D\subseteq X^\Gamma\times X^\Gamma:$ there exists a finite subset
$H$ of $\Gamma$ with $\alpha_H\subseteq D\}$ 
\end{center}
is the unique compatible uniform structure on $X^\Gamma$
(which is equipped with product topology).
Note that, the compact Hausdorff topological space $X^\Gamma$ is metrizable if and only if
$\Gamma$ is countable.
\end{convention}
\subsection*{Alexandroff compactification and Fort spaces}
For a  topological
space $Y$ choose a point $b\notin Y$ and consider the set $A(Y):=Y\cup\{b\}$ with the topology
$\{U\subseteq Y:U$ is an open subset of $Y\}\cup\{V\subseteq A(Y):Y\setminus V$
is a closed compact subset of $Y\}$, then
we call $A(Y)$ the {\it Alexandroff compactification} or {\it one point compactification of $Y$} \cite{kel, steen}.
Obviously one may study a dynamical system 
with Alexandroff compactification of a topological space as phase 
space. Now suppose $b\in F$ and consider $F$ with the topology:
\[\{U\subseteq F:b\notin U\vee(F\setminus U {\rm \: is \: finite})\}\:,\]
then we say that $F$ is a {\it Fort space (with particular point $b$)} \cite[Counterexample 24]{steen}. 
It's evident that 
the class of Fort spaces is just the class of Alexandroff compactification of discrete spaces.
\begin{convention}
From now on,  consider $F$ is a Fort space with the particular point $b$ and
$\h:F\to F$ is continuous. For $H\subseteq F$
let:
\[\gamma_H:=((F\setminus H)\times(F\setminus H))\cup\Delta_H\]
so 
\begin{center}
${\mathcal K}:=\{D\subseteq F\times F:$ there exists a finite subset
$H$ of $F\setminus\{b\}$ with $\gamma_H\subseteq D\}$ 
\end{center}
is the unique compatible uniform structure on $F$.
Note that, the compact Hausdorff topological space $F$ is metrizable if and only if
$F$ is countable. 
\end{convention}
\begin{remark}\label{fort10}
In Fort spaces $Y_1$ with particular point $b_1$ and $Y_2$ with particular point $b_2$,
$f:Y_1\to Y_2$ is continuous if and only if one of the following conditions holds true \cite[Lemma 5.1]{AD}:
\begin{itemize}
\item $f(b_1)=b_2$ and $f^{-1}(y)$ is finite, for all $y\in Y_2\setminus\{b_2\}$,
\item $f(b_1)\neq b_2$ and $Y_1\setminus f^{-1}(f(b_1))$ is finite.
\end{itemize}
\end{remark}
\section{Weak specification property}
\noindent In this section we show that $(X^\Gamma,\sigma_\varphi)$ has weak specification property
(almost weak specification property) if and only if 
$\varphi:\Gamma\to\Gamma$ does not have any periodic point.
Moreover $(X^\Gamma,\sigma_\varphi)$ has specification property,
if and only if $\varphi:\Gamma\to\Gamma$ is one--to--one without any periodic point.
On the other hand $(F,\h)$ has  weak specification property
(almost weak specification property) if and only if 
$\bigcap\{\h^n(F):n\geq1\}$ is singleton. Finally,  $(F,\h)$ has  specification property
if and only if $F$ is singleton.
\begin{lemma}\label{salam40}
If $\varphi:\Gamma\to\Gamma$ has a periodic point, then the dynamical system 
$(X^\Gamma,\sigma_\varphi)$ does not have almost weak specification property.
\end{lemma}
\begin{proof}
Suppose that $\lambda$ is a periodic point of $\Gamma$, then there exists $m\geq1$ with
$\varphi^m(\lambda)=\lambda$. Choose distinct $p,q\in X$ and consider $\alpha:=\alpha_{\{\lambda\}}\in\mathcal U$.
If $(X^\Gamma,\sigma_\varphi)$  has 
almost weak specification property, then there exists 
$N_\alpha:{\mathbb Z}_+\to{\mathbb Z}_+$ such that for all $0\leq l_1\leq k_1<l_2\leq k_2$
with $l_2-k_1\geq N_\alpha(k_2-l_2)$ there exists
$x\in X^\Gamma$ with the property:
\[(\sigma_\varphi^i(x),\sigma_\varphi^i(y_j))\in\alpha\: \forall 
j\in\{1,2\}\forall i\in\{l_j,l_j+1,\ldots,k_j\}\:,\]
for $y_1=(p)_{\theta\in\Gamma}$ and $y_2=(q)_{\theta\in\Gamma}$.
In particular for $l_1=k_1=1,l_2=N_\alpha(m)+2,k_2=N_\alpha(m)+2+m$
there exists $(x_\theta)_{\theta\in\Gamma}$ such that:
\[((x_{\varphi(\theta)})_{\theta\in\Gamma},(p)_{\theta\in\Gamma})\in\alpha\:,\]
and 
\[((x_{\varphi^i(\theta)})_{\theta\in\Gamma},(q)_{\theta\in\Gamma})\in\alpha\:
\forall i=N_\alpha(m)+2,\ldots,N_\alpha(m)+2+m\:,\]
which imply that $x_{\varphi(\lambda)}=p$ and $x_{\varphi^i(\lambda)}=q$
for all $i=N_\alpha(m)+2,\ldots,N_\alpha(m)+2+m$.
Choose $j\in\{N_\alpha(m)+2,\ldots,N_\alpha(m)+2+m\}$ such that there exists $n\geq1$
with $j=nm+1$, thus $\varphi^j(\lambda)=\varphi(\lambda)$ and 
$q=x_{\varphi^j(\lambda)}=x_{\varphi(\lambda)}=p$ which is a contradiction with 
$p\neq q$. Therefore $(X^\Gamma,\sigma_\varphi)$ 
does not have almost weak specification property.
\end{proof}
\begin{note}
For a map $\varphi:\Gamma\to\Gamma$ and $\alpha,\beta\in\Gamma$ 
we say  $\alpha\thicksim_\varphi\beta$ if there exists $i,j\geq0$
with $\varphi^i(\alpha)=\varphi^j(\beta)$. Then $\thicksim_\varphi$ is an equivalence relation
on $\Gamma$.
\end{note}
\begin{lemma}\label{salam50}
If $\varphi:\Gamma\to\Gamma$ does not have any periodic point, then the dynamical system 
$(X^\Gamma,\sigma_\varphi)$ has weak specification property.
\end{lemma}
\begin{proof}
Suppose $\varphi:\Gamma\to\Gamma$ does not have any periodic. For $\alpha\in\mathcal U$
there exists a finite subset $F$ of $\Gamma$ with $\alpha_F\subseteq\alpha$. Also, there exist
$\beta_1,\ldots,\beta_m\in\Gamma$  and $N\geq1$  such that
$F\subseteq\bigcup\{\varphi^{-i}(\beta_j):j=1,\ldots,m,i=0,\ldots,N\}$ and
for all distinct $s,t\in\{1,\ldots,m\}$,
$\beta_s\not\thicksim_\varphi\beta_t$. Now consider $x_1=(x^1_\theta)_{\theta\in\Gamma},
\ldots,x_k=(x^k_\theta)_{\theta\in\Gamma}\in X^\Gamma$ and
$0\leq i_1\leq j_1<i_2\leq j_2<\cdots<i_k\leq j_k$ with $i_{s+1}-j_s\geq N+1$
for all $s\in\{1,\ldots,k-1\}$. Choose $p\in X$ and define 
$z=(z_\theta)_{\theta\in\Gamma}\in X^\Gamma$ in the following way:
\[ z_\theta=\left\{\begin{array}{lc}
x^s_{\varphi^t(\gamma)} & 1\leq s\leq k\wedge i_s\leq t\leq j_s\wedge \gamma\in F \wedge \theta=\varphi^t(\gamma) \\
p & {\rm otherwise}\end{array}
\right.\]
First of all note that $z_\theta$s are well-defined, otherwise there exist $\gamma_1,\gamma_2\in F$,
$s_1<s_2$, and $t_1<t_2$ with $i_{s_1}\leq t_1\leq j_{s_1}<i_{s_2}\leq t_2\leq  j_{s_2}$ 
with
$\varphi^{ t_1}(\gamma_1)=\varphi^{ t_2}(\gamma_2)$,
 thus $\gamma_1\thicksim_\varphi\gamma_2$ and there exists $i\in\{1,\ldots,m\}$
 and $r_1,r_2\in\{0,\ldots,N\}$ with $\varphi^{r_1}(\gamma_1)=
 \varphi^{r_2}(\gamma_2)=\beta_i$, so by $i_{s_2}-j_{s_1}\geq N+1$,
 $t_2>N$ and $\varphi^{ t_1}(\gamma_1)=\varphi^{ t_2}(\gamma_2)=
 \varphi^{ t_2-r_2}(\beta_i)=\varphi^{ t_2-r_2+r_1}(\gamma_1)$. By  $ t_2-r_2+r_1\geq t_2-N\geq i_{s_2}-N>j_{s_1}\geq t_1$, we have
 $ t_2-r_2+r_1>t_1$ and $\varphi^{ t_1}(\gamma_1)$ is a periodic point of $\varphi$,
 which is a contradiction, so $z_\theta$s are well-defined.
 \\
It is easy to see that for all $s\in\{1,\ldots,k\}$, $\gamma\in F$ and 
$i_s\leq t\leq j_s$ we have $x^s_{\varphi^t(\gamma)}=z_{\varphi^t(\gamma)}$, i.e.
$(\sigma_\varphi^t(z),\sigma_\varphi^t(x^s))\in\alpha_F\subseteq\alpha$,
which completes the proof.
\end{proof}
\begin{theorem}\label{salam60}
The following statements are equivalent:
 \begin{itemize}
 \item $(X^\Gamma,\sigma_\varphi)$ has weak specification property,
 \item $(X^\Gamma,\sigma_\varphi)$ has almost weak specification property,
 \item  $\varphi:\Gamma\to\Gamma$ does not have any periodic point.
 \end{itemize}
\end{theorem}
\begin{proof}
Use Lemmas~\ref{salam40} and ~\ref{salam50}.
\end{proof}
\begin{corollary}\label{salam70}
The following statements are equivalent:
 \begin{itemize}
 \item[1.] $(X^\Gamma,\sigma_\varphi)$ has specification property,
  \item[2.]  $\varphi:\Gamma\to\Gamma$ is one--to--one without any periodic point.
 \end{itemize}
\end{corollary}
\begin{proof}
``(1) $\Rightarrow$ (2)'' Suppose $(X^\Gamma,\sigma_\varphi)$ 
has specification property, then for all $x\in X^\Gamma$
and open neighbourhood $U$ of $x$ there exists $\alpha\in\mathcal U$ with
$\alpha[x]\subseteq U$. Since
$(X^\Gamma,\sigma_\varphi)$ 
has specification property, there exists a periodic point $z$ of $\sigma_\varphi$ with
$(x,z)\in\alpha$, i.e. $z\in\alpha[x]\subseteq U$ and the collection of all
periodic points of $\sigma_\varphi$ is dense in $X^\Gamma$. By \cite[Theorem 2.6]{ANT13}
$\varphi:\Gamma\to\Gamma$ is one--to--one and using 
Lemma~\ref{salam40} $\varphi:\Gamma\to\Gamma$ does not have
any periodic point.
\\
``(2) $\Rightarrow$ (1)'' Suppose  $\varphi:\Gamma\to\Gamma$ is one--to--one 
without any periodic point. By using Theorem~\ref{salam60}, we get that  
$(X^\Gamma,\sigma_\varphi)$ has weak specification property. Consider
$\alpha\in\mathcal U$, there exists $\beta\in\mathcal U$ with
$\beta^{-1}\circ\beta\subseteq\alpha$, moreover since
$(X^\Gamma,\sigma_\varphi)$ has weak specification property
there exists $N\geq1$
such that for all $y_1,\ldots,y_n\in  X^\Gamma$ ($n\geq2$) and
$0\leq l_1\leq k_1<l_2\leq k_2<\ldots<l_n\leq k_n$ with 
$l_{i+1}-k_i\geq N$ 
(for all $i=1,\ldots,n-1$) there exists $x\in X^\Gamma$
such that 
\[(\sigma_\varphi^i(x),\sigma_\varphi^i(y_j))\in\beta\: 
\forall j\in\{1,\ldots,n\}\forall i\in\{l_j,l_j+1,\ldots,k_j\}\:.\]
Since $\sigma_\varphi:X^\Gamma\to X^\Gamma$ is uniformly continuous,
there exists $\mu\in\mathcal U$ with:
\[(u,v)\in\mu\Rightarrow(\sigma_\varphi^i(u),\sigma_\varphi^i(v))\in\beta \:(\forall
i\in\{0,1,\ldots,k_n\}\:.\]
Since  $\varphi:\Gamma\to\Gamma$ is one--to--one, the collection of all
periodic points of $\sigma_\varphi$ is dense in $X^\Gamma$ by \cite[Theorem 2.6]{ANT13}.
Hence
there exists a periodic point $z$ of $\sigma_\varphi$ with $(x,z)\in\mu$, thus
$(\sigma_\varphi^i(x),\sigma_\varphi^i(z))\in\beta$ for all $i\in\{0,1,\ldots,k_n\}$.
Therefore $(\sigma_\varphi^i(z),\sigma_\varphi^i(y_j))\in
\beta^{-1}\circ\beta\subseteq\alpha$ for all $ j\in\{1,\ldots,n\}$
and $ i\in\{l_j,l_j+1,\ldots,k_j\}$,
which completes the proof.
\end{proof}
\begin{theorem}\label{AF10}
In the dynamical system $(F,\h)$ the following statements are equivalent:
\begin{itemize}
\item[a.] the dynamical system $(F,\h)$ has  weak specification property,
\item[b.] the dynamical system $(F,\h)$ has almost weak specification property,
\item[c.] $\bigcap\{\h^n(F):n\geq1\}$ is singleton.
\end{itemize}
\end{theorem}
\begin{proof}
It's evident that (a) implies (b).
\\
(b) $\Rightarrow$ (c): Suppose $(F,\h)$ has almost weak specification property. Note that
$\bigcap\{\h^n(F):n\geq1\}\neq\varnothing$ since $F$ is a nonempty compact Hausdorff and
$\h:F\to F$ is continuous. If $\bigcap\{\h^n(F):n\geq1\}$ is not singleton, then there exist distinct points
$z,y\in\bigcap\{\h^n(F):n\geq1\}$. We may suppose $z\neq b$.
There exists $N_{\gamma_{\{z\}}}:
{\mathbb Z}_+\to{\mathbb Z}_+$ such that ${\displaystyle\lim_{n\to\infty}\frac{N_{\gamma_{\{z\}}}(n)}{n}
}=0$ and for all $y_1,\ldots,y_n\in F$ and $0\leq l_1\leq k_1<l_2\leq k_2<\ldots<l_n\leq k_n$ with 
$l_{i+1}-k_i\geq N_{\gamma_{\{z\}}}(k_{i+1}-l_{i+1})$ (for all $i=1,\ldots,n-1$) there exists $x\in Z$
such that 
\[(\h^i(x),\h^i(y_j))\in\gamma_{\{z\}}\:\forall j\in\{1,\ldots,n\}\forall i\in\{l_j,l_j+1,\ldots,k_j\}\:.\]
Let:
\[l_1=k_1=0,l_2=k_2\geq N_{\gamma_{\{z\}}}(0)\:,\]
Choose $w\in\h^{-(l_2+1)}(z)$, hence there exists $x_1,x_2\in F$ such that:
\begin{align}\label{eq1 lm1}
(x_1,z)=(\h^{l_1}(x_1),\h^{l_1}(z)),(\h^{l_2}(x_1),z)=(\h^{l_2}(x_1),\h^{l_2}(\h(w)))\in\gamma_{\{z\}}\:,
\end{align}
and:
\begin{align}\label{eq2 lm2}
(x_2,z)=(\h^{l_1}(x_2),\h^{l_1}(z)),(\h^{l_2}(x_2),\h^{l_2}(w))\in\gamma_{\{z\}}\:.
\end{align}
By \eqref{eq1 lm1}, $x_1=z=\h^{l_2}(x_1)$, thus $z=\h^{l_2}(z)$.
Considering \eqref{eq2 lm2}
we have $x_2=z$ and therefore 
$(z,\h^{l_2}(w))=(\h^{l_2}(z),\h^{l_2}(w))=(\h^{l_2}(x_2),\h^{l_2}(w))\in\gamma_{\{z\}}$,
$z=\h^{l_2}(w)$, which leads to $\h(z)=\h^{l_2+1}(w)=z$.
\\
Now for $y\in\bigcap\{\h^n(F):n\geq1\}\setminus\{z\}$ and choose $v\in\h^{-l_2}(y)$, 
there exists $x_3\in F$
such that:
\[(x_3,z)=(\h^{l_1}(x_3),\h^{l_1}(z)),(\h^{l_2}(x_3),y)=(\h^{l_2}(x_3),\h^{l_2}(v))\in\gamma_{\{z\}}\:,\]
thus $x_3=z$ and $(z,y)=(\h^{l_2}(z),y)=(\h^{l_2}(x_3),\h^{l_2}(v))\in\gamma_{\{z\}}$
which shows that $y=z$. Therefore $\bigcap\{\h^n(F):n\geq1\}=\{z\}$ which is in contradiction
with $y\neq z$. Thus $\bigcap\{\h^n(F):n\geq1\}$ is singleton.
\\
(c) $\Rightarrow$ (a): Suppose $\bigcap\{\h^n(F):n\geq1\}$ is singleton:
\\
{\it Case I:} $\bigcap\{\h^n(F):n\geq1\}=\{b\}$. Let $\alpha\in{\mathcal K}$. There exists finite
subset $A$ of $F\setminus\{b\}$ with $\gamma_A\subseteq\alpha$. For all
$y\in A$ there exists $m_y\geq1$ with $y\notin\h^{m_y}(F)$, thus $\h^{-m_y}(y)=\varnothing$.
Let $m=\mathop{\Sigma}\limits_{y\in A}m_y$, then $\h^{-m}(A)=\varnothing$.
So for all $n\geq m$ , $\h^n(F)\cap A=\varnothing$. Let $N_\alpha=m$, then
for all 
$y_1,\ldots,y_n\in F$ and $0\leq l_1\leq k_1<l_2\leq k_2<\ldots<l_n\leq k_n$ with 
$l_{i+1}-k_i\geq m$ (for all $i=1,\ldots,n-1$),
$m\leq l_2$. Thus for all $i\geq l_2(\geq m)$ and $j\geq2$,
$\h^i(y_1),\h^i(y_j)\notin A$, so $(\h^i(y_1),\h^i(y_j))\in\gamma_A\subseteq\alpha$.
On the other hand for $0\leq i\leq k_1$, $(\h^i(y_1),\h^i(y_1))\in\gamma_A\subseteq\alpha$.
Hence 
\[(\h^i(y_1),\h^i(y_j))\in\alpha\:\forall j\in\{1,\ldots,n\}\forall i\in\{l_j,l_j+1,\ldots,k_j\}\:,\]
which completes the proof.
\\
{\it Case II:} $\bigcap\{\h^n(F):n\geq1\}=\{c\}$ and $c\neq b$. In this case $\h(b)\neq b$
and $\h(F)$ is finite, so by $\bigcap\{\h^n(F):n\geq1\}=\{c\}$, there exists $m\geq1$ such that:
\[\forall n\geq m\:(\h^n(F)=\{c\})\:.\]
Now for all 
$y_1,\ldots,y_n\in F$ and $0\leq l_1\leq k_1<l_2\leq k_2<\ldots<l_n\leq k_n$ with 
$l_{i+1}-k_i\geq m$ (for all $i=1,\ldots,n-1$), $l_2\geq m$
and $\h^i(y_j)=\h^i(y_1)$ for all $i\in\{l_j,\ldots,k_j\}$ and $j\in\{1,\ldots,n\}$.
Thus $(F,\h)$ has  weak specification property.
\end{proof}
\begin{theorem}\label{AF20}
The dynamical system $(F,\h)$ has  specification property,
if and only if $F=\{b\}$.
\end{theorem}
\begin{proof}
It's clear that for $F=\{b\}$, $(F,\h)$ has  specification property.
\\
Conversely, suppose $(F,\h)$ has  specification property, then it has weak specification property.
By Theorem~\ref{AF10}, $\bigcap\{\h^n(F):n\geq1\}$ is singleton. For
$z\in F\setminus\{b\}$ there exists a periodic point $x\in F$ such that for
$l_1=0$, $(x,z)=(\h^{l_1}(x),\h^{l_1}(z))\in\gamma_{\{z\}}$ so $z=x$
and $z$ is a periodic point of $\h$. Thus 
\[F\setminus\{b\}\subseteq Per(\h)\subseteq 
\bigcap\{\h^n(F):n\geq1\}\: ,\]
hence $F\setminus\{b\}$ has at most one element.
Suppose that $F=\{x,b\}$. We may consider the following cases, for $x\neq b$:
\\
{\it Case I:} $\bigcap\{\h^n(F):n\geq1\}=\{x\}\neq \{b\}$. In this case $\h(x)=\h(b)=x$. So
$Per(f)=\{x\}$, thus for $l_1=0$ , $(x,b)=(\h^{l_1}(x),\h^{l_1}(b))\in\gamma_{\{x\}}$
which leads to the contradiction $x=b$, hence this case does not occur.
\\
{\it Case II:} $\bigcap\{\h^n(F):n\geq1\}=\{b\}\neq \{x\}$. In this case we have the contradiction
$x\in F\setminus\{b\}\subseteq Per(\h)\subseteq 
\bigcap\{\h^n(F):n\geq1\}=\{b\}$
hence this case does not occur too.
\\
Considering the above cases, we get $x=b$ and $F=\{b\}$.
\end{proof}
\section{Stroboscopical property}
\noindent In this section we show that $(X^\Gamma,\sigma_\varphi)$ has uniform stroboscopical property,
 (stroboscopical property) if and only if $\varphi:\Gamma\to\Gamma$ is one--to--one. Also
 $(F,\h)$ has uniform stroboscopical property
(stroboscopical property) if and only if all points of $F$ are periodic points of $\h$.
\begin{lemma}\label{salam80}
If $(X^\Gamma,\sigma_\varphi)$ has stroboscopical property, then
$\varphi:\Gamma\to\Gamma$ is one--to--one.
\end{lemma}
\begin{proof}
Suppose $\varphi:\Gamma\to\Gamma$ is not one--to--one, then there are distinct points 
$\beta,\lambda\in \Gamma$ such that $\varphi(\beta)=\varphi(\lambda)$. 
Choose distinct $p,q\in X$.
Consider $z=(z_\alpha)_{\alpha\in\Gamma}\in X^\Gamma$ with:
\[z_\alpha=\left\{\begin{array}{lc} p & \alpha=\beta\:, \\ q & {\rm otherwise}\:. \end{array}\right.\]
then for all $(x_\alpha)_{\alpha\in\Gamma}\in X^\Gamma$ and $n\geq1$
, $x_{\varphi^n(\beta)}=x_{\varphi^n(\lambda)}$. Thus there is not any sequence
$\{n_i\}_{i\geq0}$ in $\mathbb N$ with ${\displaystyle\lim_{i\to\infty}
\sigma_\varphi^{n_i}((x_\alpha)_{\alpha\in\Gamma})}=z$. Therefore 
$(X^\Gamma,\sigma_\varphi)$ does not have stroboscopical property.
\end{proof}
\begin{lemma}\label{salam110}
If $A=\{m^1_i\}_{i\geq0}$ is a strictly increasing sequence in $\mathbb N$, then:
\begin{itemize}
\item[1.] there exists a map $f:{\mathbb N}\to{\mathbb N}$, and $A$ has a subsequence
 	 $\{n_i\}_{i\geq0}$ 
 	such that for all $m\geq1$ and $i\geq m$, $n_i\equiv f(m) \: ({\rm mod}\: m)$
 	and $f(m)\leq m$;
 \item[2.] $A$ has a subsequence
$\{m_i\}_{i\geq0}$ such that $m_{i+1}-m_i>2i$ for all $i\geq0$.
 \end{itemize}
\end{lemma}
\begin{proof}
{\bf 1.}
Suppose $A=\{m^1_i\}_{i\geq0}=:A_1$ and $k_1=1$, then there exists $k_2\in\{1,2\}$
and $A$ has a subsequence  
$A_2=\{m^2_i\}_{i\geq0}$ such that for all $i\geq0$, $m^2_i\equiv k_2 ({\rm mod}\:2)$.
For $s\geq2$, suppose that we have chosen subsequences $A_1=\{m^1_i\}_{i\geq0},\ldots,
A_s=\{m^s_i\}_{i\geq0}$ and natural numbers $k_1,\ldots,k_s$. Then there exists 
$k_{s+1}\in\{1,\ldots,s+1\}$,
and $A_s$ has a subsequence
$A_{s+1}=\{m^{s+1}_i\}_{i\geq0}$ such that for all $i\geq0$, 
$m^{s+1}_i\equiv k_{s+1} ({\rm mod}\:s+1)$. Let $f:\mathbb{N}\to\mathbb{N}$
with $f(i)=k_i$ ($i\geq1$) and subsequence  $\{n_i\}_{i\geq0}$ of $A$ with
$n_0=m_0^1$ and $n_i=m_i^i$ for $i\geq1$, then the map $f:\mathbb{N}\to\mathbb{N}$
and subsequence $\{n_i\}_{i\geq0}$ of $A$ satisfies (1).
\\
{\bf 2.} We construct $\{m_i\}_{i\geq0}$ inductively. Let $m_0:=m^1_0$, for
$i\geq0$ suppose we have chosen $m_i=m^1_{n_i}$ let $m_{i+1}:=m^1_{n_i+2i+1}$.
The the subsequence $\{m_i\}_{i\geq0}$ of $A$ satisfies (2).
\end{proof}
\begin{lemma}\label{salam90}
If all points of $\Gamma$ are periodic points of $\varphi:\Gamma\to\Gamma$,
then $(X^\Gamma,\sigma_\varphi)$ has uniform stroboscopical property.
\end{lemma}
 \begin{proof}
 Let $A$ be a strictly increasing sequence in $\mathbb N$. 
By Lemma~\ref{salam110} there exists a map $f:{\mathbb N}\to{\mathbb N}$
and $A$ has a subsequence
 $\{n_i\}_{i\geq0}$
 such that for all $m\geq1$ and $i\geq m$, $n_i\equiv f(m) \: ({\rm mod}\: m)$
 and $f(m)\leq m$. For all 
 $\alpha\in\Gamma$ let $k_\alpha=\min\{n\geq1:\varphi^n(\alpha)=\alpha\}$.
 For $z=(z_\alpha)_{\alpha\in\Gamma}\in X^\Gamma$ define  $z^*=(z^*_\alpha)_{\alpha\in\Gamma}\in X^\Gamma$
 with:
 \[z^*_\alpha=z_{\varphi^{k_\alpha-f(k_\alpha)}(\alpha)}\:\:\:(\alpha\in\Gamma)\:.\] 
 Then for all
 $\alpha\in\Gamma$ and $i\geq f(k_\alpha)$,  $k_\alpha-f(k_\alpha)+n_i\equiv0\:
 ({\rm mod}\:k_\alpha)$, so 
 \[\forall\alpha\in\Gamma\:\forall i\geq f(k_\alpha)\:\:(\varphi^{k_\alpha-f(k_\alpha)+n_i}(\alpha)=\alpha\:\:.)\]
 For $D\in\mathcal{U}$ there exists finite subset $H=\{\beta_1,\ldots,\beta_s\}$ of $\Gamma$ with $\alpha_H\subseteq D$.
 For all $i\geq\max\{f(k_{\beta_1}),\ldots,f(k_{\beta_m})\}$ and for all $z=(z_\alpha)_{\alpha\in\Gamma}\in X^\Gamma$
 we have $(z,\sigma_\varphi^{n_i}(z^*))\in\alpha_H\subseteq D$.
 Thus for $\mathop{\varrho:X^\Gamma\to X^\Gamma}\limits_{\SP \SP z\mapsto z^*}$ the sequence
 $\{\sigma_\varphi^{n_i}\circ\varrho\}_{i\geq1}$ converges uniformly to $id_{X^\Gamma}$.
 \end{proof}
\begin{remark}\label{salam120}
If $\varphi:\Gamma\to\Gamma$ is one--to--one, then each $D\in\frac{\Gamma}{\thicksim_\varphi}
(=\{\frac{\alpha}{\thicksim_\varphi}:\alpha\in\Gamma\})$ satisfies exactly one of the 
following conditions:
\begin{itemize}
\item $D$ has finite elements like $m$ and for all $\alpha\in D$, $D=\{\varphi^i(\alpha):
	1\leq i\leq m\}$,
\item $D$ is infinite and there exists $\alpha\in\Gamma$ with $D=\{\varphi^n(\alpha):n\geq0\}$,
\item $D$ is infinite and we may suppose that $D=\{\beta_n:n\in\mathbb{Z}\}$ with $\varphi(\beta_n)=
	\beta_{n+1}$ for all $n\in\mathbb{Z}$.
\end{itemize}
\end{remark}
\begin{lemma}\label{salam100}
If $\varphi:\Gamma\to\Gamma$ is one--to--one and does not have any periodic point,
then $(X^\Gamma,\sigma_\varphi)$ has uniform stroboscopical property.
\end{lemma}
  \begin{proof}
  Suppose $\varphi:\Gamma\to\Gamma$ is a one--to--one map without any periodic point.
  Using axiom of choice there exists a subset of $\Gamma$ like $\Lambda$ such that
  $\bigcup\{\frac{\alpha}{\thicksim_\varphi}:\alpha\in\Lambda\}=\Gamma$ and for all
  distinct $\alpha,\beta\in\Lambda$, $\frac{\alpha}{\thicksim_\varphi}\cap 
  \frac{\beta}{\thicksim_\varphi}=\varnothing$, also we may suppose each $\alpha\in\Lambda$
  provides exactly one of the following conditions (use Remark~\ref{salam120} too):
  \begin{itemize}
  \item $\frac{\alpha}{\thicksim_\varphi}=\{\varphi^n(\alpha):n\geq0\}$,
  \item $\forall n\geq0 \exists\beta\in\Gamma\:(\varphi^n(\beta))=\alpha)$.
  \end{itemize}
  Let 
  \[\Lambda_1:=\{\alpha\in\Lambda:\frac{\alpha}{\thicksim_\varphi}=\{\varphi^n(\alpha):n\geq0\}\}\]
  and 
  \[\Lambda_2:=\{\alpha\in\Lambda:\forall n\geq0\:(\varphi^{-n}(\alpha)\neq\varnothing)\}\:.\]
  Moreover for all $\alpha\in\Gamma$, choose 
  $\lambda_\alpha\in\Lambda\cap\frac{\alpha}{\thicksim_\varphi}$. 
  Suppose  $A$ is a strictly increasing sequence in $\mathbb N$. 
By Lemma~\ref{salam110}, $A$ has a subsequence
 like $\{n_i\}_{i\geq0}$ such that 
 \linebreak
 $n_{i+1}-n_i>2i$ for all $i\geq0$. Consider  $p\in X$ and  for
 $z=(z_\alpha)_{\alpha\in\Gamma}\in X^\Gamma$ define $z^*=(z^*_\alpha)_{\alpha\in\Gamma}\in X^\Gamma$ with:
 \[z^*_\alpha=\left\{\begin{array}{lc}
 & \\
 z_{\varphi^j(\lambda_\alpha)} & 0\leq j<n_{i+1}-n_i , \alpha=\varphi^{n_i+j}(\lambda_\alpha) , \lambda_\alpha\in\Lambda_1 \:, \\
 & \\
 z_{\varphi^j(\lambda_\alpha)} & -i< j< i , \alpha=\varphi^{n_i+j}(\lambda_\alpha) , \lambda_\alpha\in\Lambda_2\:, \\
 & \\
 p & {\rm otherwise}\:. \\
 & 
 \end{array}\right.\]
 Note that using $0<n_1<n_2<n_3<\cdots$, for $k\geq n_1$ there exists unique $i\geq1$ and $j\in\{0,\ldots,n_{i+1}-n_i-1\}$
 with $k=n_i+j$ thus for $\mu\in\Lambda_1$  and
 $\alpha\in\{\varphi^n(\mu):n\geq n_1\}$, $z^*_\alpha$ is well-defined.
  Moreover for $\mu\in\Lambda_2$ if there exist
  $i_1,i_2\geq1$, $j_1\in\{-(i_1-1),\ldots,i_1-1\}$ and $j_2\in\{-(i_2-1),\ldots,i_2-1\}$ with
  $\varphi^{n_{i_1}+j_1}(\mu)=\varphi^{n_{i_2}+j_2}(\mu)$ we have $n_{i_1}+j_1=n_{i_2}+j_2$.
  If $i_1\neq i_2$ we may suppose $i_1<i_2$ thus $2(i_2-1)<n_{i_2}-n_{i_2-1}\leq n_{i_2}-n_{i_1}=j_1-j_2\leq i_1-1+i_2-1$
  which leads to contradiction $i_2< i_1$, thus $i_2=i_1$ and $j_2=j_1$. Therefore $z^*_\alpha$s are
  well-defined. We have:
  \[\forall\mu\in\Lambda_1\:\forall j\geq0\:\forall i\geq j+2\:(z^*_{\varphi^{n_i+j}(\mu)}=z_{\varphi^j(\mu)})\:,\tag{i}\]
  since for $\mu\in\Lambda_1$, $j\geq0$ and $i\geq j+2$ the relation $\lambda_{\varphi^{n_i+j}(\mu)}=\mu$
  and inequalities $0\leq j<2(j+2)\leq 2i< n_{i+1}-n_i$ lead to
  $z^*_{\varphi^{n_i+j}(\mu)}=z_{\varphi^j(\mu)}$. 
  \\
  Also
\[\forall\mu\in\Lambda_2\:\forall j\in\mathbb{Z}\:\forall i>|j|\:(z^*_{\varphi^{n_i+j}(\mu)}=z_{\varphi^j(\mu)})\:.\tag{ii}\] 
For $D\in\mathcal{U}$ there exists finite subset $H=\{\beta_1,\ldots,\beta_s\}$ of $\Gamma$ with $\alpha_H\subseteq D$.
There exists $N\geq1$, finite subset $H_1$ of $\Lambda_1$ and finite subset $H_2$ of $\Lambda_2$ with
$H\subseteq\bigcup\{\varphi^t(H_1\cup H_2):|t|\leq N\} (=(\bigcup\{\varphi^t(H_1):0\leq t\leq N\})\cup(\bigcup\{\varphi^t(H_2):-t\leq t\leq N\}))$. 
Using (i) and (ii) for all $z=(z_\alpha)_{\alpha\in\Gamma}\in X^\Gamma$ we have:
\[\forall\theta\in H\: \forall i\geq N+2\:(z^*_{\varphi^{n_i}(\theta)}=z_\theta)\:, \]
 thus $(z,\sigma_\varphi^{n_i}(z^*))\in\alpha_H\subseteq D$ for all $z=(z_\alpha)_{\alpha\in\Gamma}\in X^\Gamma$
 and $i\geq N+2$.
Hence for $\mathop{\varrho:X^\Gamma\to X^\Gamma}\limits_{\SP \SP z\mapsto z^*}$ the sequence
 $\{\sigma_\varphi^{n_i}\circ\varrho\}_{i\geq1}$ converges uniformly to $id_{X^\Gamma}$.
  \end{proof}
\begin{theorem}\label{salam200}
The dynamical system 
 $(X^\Gamma,\sigma_\varphi)$ has uniform stroboscopical property if and only if 
 $\varphi:\Gamma\to\Gamma$ is one-to-one.
\end{theorem}
  \begin{proof}
  If  $(X^\Gamma,\sigma_\varphi)$ has uniform stroboscopical property, then it has stroboscopical property and by Lemma~\ref{salam80}
  the map $\varphi:\Gamma\to\Gamma$ is one--to--one.
  \\
  Now suppose $\varphi:\Gamma\to\Gamma$ is one--to--one. If either all points
  of $\Gamma$ are periodic points of $\varphi$ or $\varphi$ does not have any periodic
  point, then we are done by Lemmas~\ref{salam90} and~\ref{salam100}. Otherwise
  $Per(\varphi)$ (the collection of all periodic points of $\varphi$) and 
  $\Gamma\setminus Per(\varphi)$ are two nonempty subsets of $\Gamma$.
  By Lemmas~\ref{salam90} and~\ref{salam100} both dynamical systems
$(X^{Per(\varphi)},\sigma_{\varphi\restriction_{Per(\varphi)}})$ and
$(X^{\Gamma\setminus Per(\varphi)},
\sigma_{\varphi\restriction_{\Gamma\setminus Per(\varphi)}})$ have
uniform stroboscopical property. In order to complete the proof use the fact that if both
dynamical systems $(Z,f)$ and $(Y,h)$ have uniform stroboscopical property, then
$(Z\times Y,f\times h)$ (where $(f\times h)(u,v)=(f(u),h(v))$ for $(u,v)\in Z\times Y$)
has uniform stroboscopical property.
  \end{proof}
\begin{note} 
Using Lemmas~\ref{salam80} and Theorem~\ref{salam200},
for finite discrete topological space $X$ with at least two elements,
 nonempty set $\Gamma$, and arbitrary map $\varphi:\Gamma\to\Gamma$ the following
 statements are equivalent:
 \begin{itemize}
 \item  $(X^\Gamma,\sigma_\varphi)$ has uniform stroboscopical property,
 \item  $(X^\Gamma,\sigma_\varphi)$ has stroboscopical property,
 \item $\varphi:\Gamma\to\Gamma$ is one--to--one.
 \end{itemize}
\end{note}
\begin{note}
For countable $\Gamma$, by \cite[Proposition 5]{LS03} and \cite{ANT13}, 
$(X^\Gamma,\sigma_\varphi)$ has the strong stroboscopical property
if and only if $\varphi:\Gamma\to\Gamma$ is one--to--one without periodic points.
\end{note}
\begin{theorem}\label{AF30}
The following statements are equivalent:
\begin{itemize}
\item[a.] $(F,\h)$ has uniform stroboscopical property,
\item[b.] $(F,\h)$ has stroboscopical property,
\item[c.] $Per(\h)=F$.
\end{itemize}
\end{theorem}
\begin{proof} (b) $\Rightarrow$ (c).
Suppose $(F,\h)$ has stroboscopical property and $z$  is an isolated point of $F$,
then there exists $x\in F$ and $\{n\}_{n\geq0}$ has a subsequence like $\{n_k\}_{k\geq0}$ such that
${\displaystyle\lim_{k\to\infty}\h^{n_k}(x)}=z$, hence there exists $N\geq1$ such that:
\[\forall k\geq N\:(\h^{n_k}(x)=z)\:,\]
in particular $\h^{n_N}(x)=z=\h^{n_{N+1}}(x)$, thus $\h^{n_{N+1}-n_N}(z)=z$ and $z\in Per(\h)$.
\\
On the other hand, if $z$  is limit point of $F$, then $F$ is infinite and $z=b$. Since
$F\setminus\{b\}\subseteq Per(\h)$, $\h\restriction_{F\setminus\{b\}}:F\setminus\{b\}\to F$
is one-to-one which leads to $\h(b)=b$ by Remark~\ref{fort10}.
\\
(c) $\Rightarrow$ (a).
Suppose $Per(\h)=F$ and $A$ is an arbitrary
increasing sequence in $\mathbb N$. The subsequences $A_1,A_2,\ldots$ of  
of $A$ defined inductively as follows:
\begin{itemize}
\item let $A_1:=A=\{n^1_k\}_{k\geq0}$
	and $r_1=0$, 
\item suppose for 
$t\geq1$, $A_t=\{n^t_k\}_{k\geq0}$ is a subsequence of $A$ and $r_t\in\{0,\ldots,t-1\}$ such that 
for all $k\geq0$, $n^t_k\mathop{\equiv}\limits^{\mod t} r_t$; 
then there exists $r_{t+1}\in\{0,\ldots,t\}$ and subsequence
$A_{t+1}=\{n^{t+1}_k\}_{k\geq0}$ of $A_t$ such $n^{t+1}_k\mathop{\equiv}\limits^{\mod t+1} r_{t+1}$.
\end{itemize}
Then $B:=\{n_k^k\}_{k\geq1}$ is a subsequence of $A$ with
\[\forall t\geq1 \forall k\geq t\:(n_k^k\mathop{\equiv}\limits^{\mod t} r_t)\:.\]
For $z\in F$ let $m_z:=\min\{m\geq1:\h^m(z)=z\}$. Consider
$\varrho:F\to F$ with $\varrho(z)=\h^{m_z-r_{m_z}}(z)$ ($z\in F$). Then for all
$z\in Z$ and $k\geq m_z$, $\h^{n_k^k}(\varrho(z))=z$. For $\alpha\in \mathcal{K}$
there exists a finite subset $A$ of $F\setminus\{b\}$ with $\gamma_A\subseteq\alpha$.
Let $N=\max(\{m_z:z\in A\}\cup\{m_b\})$, for all $k\geq N$ and $z\in F$.
\begin{itemize}
\item $m_z\leq N$: in this case $\h^{n_k^k}(\varrho(z))=z$, and $(\h^{n_k^k}(\varrho(z)),z)=(z,z)\in\alpha$.
\item $m_z> N$: note that for all $i\geq0$, $m_z=m_{\h^i(z)}$, thus for all
	$n\geq N$ we have $m_{\h^{n_k^k}(\varrho(z))}=m_{\h^{n_k^k+m_z-r_{m_z}}(z)}=m_z>N$
	therefore $z, \h^{n_k^k}(\varrho(z))\notin A$ and $(\h^{n_k^k}(\varrho(z)),z)
	\in\gamma_A\subseteq\alpha$.
\end{itemize}
Regarding the above two cases we have $(\h^{n_k^k}(\varrho(z)),z)=(\h^{n_k^k}(\varrho(z)),id_F(z))\in 
\alpha$ for all $k\geq N$. Thus $\{\h^{n_k^k}\circ\varrho\}_{k\geq1}$ converges uniformly
to $id_F:F\to F$ and $(F,\h)$ has uniform stroboscopical property.
\end{proof}
\begin{theorem}
The dynamical system $(F,\h)$ has strongly stroboscopical property if and only if $F=\{b\}$.
\end{theorem}
\begin{proof}
Suppose $(F,\h)$ has strongly stroboscopical property, if $F\neq\{b\}$, then $F$ has at least two
(distinct) isolated points like $x,y$. There exists $m\geq1$ with $\h^m(x)=x$, by
Theorem~\ref{AF30}, since $(F,\h)$ has strongly stroboscopical property and $x$
is an isolated point of $F$, $y\in\omega_\h(\{mn\}_{n\geq0},x)=\{x\}$ which is
in contradiction with $y\neq x$, thus $F=\{b\}$.
\end{proof}
\section{Two diagrams}
\noindent Let's recall that inthe paper,  $X$ is a finite discrete space with at least two
elements, $\Gamma$ is nonempty, $\varphi:\Gamma\to\Gamma$ is  an arbitrary map, $F$
is a Fort space with particular point $b$, and $\h:F\to F$ is continuous. So according to the previous
sections, we get the following table:
{\scriptsize \begin{center}
\begin{tabular}{l|c|c|}
	\begin{tabular}{lr}
	& $(Z,f)$ 
	\\ \hline
	$\rho\SP\SP\SP\SP\SP$ & 
	\end{tabular} & $(X^\Gamma,\sigma_\varphi)$ & $(F,\h)$  \\ \hline 
	\begin{tabular}{l}
	almost weak specification \\ \hline
	weak specification
	\end{tabular}
& $Per(\varphi)=\varnothing$ & $\bigcap\{\h^n(F):n\geq1\}$ is singleton \\ \hline
	\begin{tabular}{l}
	uniform stroboscopical \\ \hline
	stroboscopical
	\end{tabular}
& $\varphi$ is one-to-one & $Per(\h)=F$ \\ \hline
	\begin{tabular}{l}
	specification \\ \hline
	strongly stroboscopical
	\end{tabular}
& \begin{tabular}{c} $Per(\varphi)=\varnothing$  and \\ $\varphi$ is one-to-one \end{tabular}
& $F=\{b\}$ \\ \hline \end{tabular}
\\ $\:$ \\
{\bf Table A}
\end{center}}
\noindent Where the dynamical system $(Z,f)$ has $\rho$ property if and only if it has the phrase
occurs in the corresponding case.
\\
Using Table A, we have the following diagrams:
{\scriptsize \begin{center}

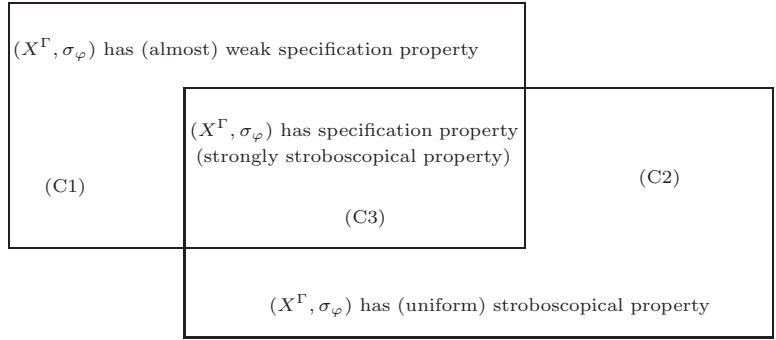
\unitlength .58mm 
\linethickness{0.4pt}
\ifx\plotpoint\undefined\newsavebox{\plotpoint}\fi 
\begin{picture}(178.75,79.5)(0,0)
\put(4,23.5){\framebox(118,56)[]{}}
\put(44.25,3){\framebox(134.5,57)[]{}}
\put(56.5,69){\makebox(0,0)[cc]
{$\:\:\:(X^\Gamma,\sigma_\varphi)$ has (almost) weak specification property}}
\put(114,10.25){\makebox(0,0)[cc]
{$(X^\Gamma,\sigma_\varphi)$ has (uniform) stroboscopical property}}
\put(85.5,30.25){\makebox(0,0)[cc]{(C3)}}
\put(16.75,37.25){\makebox(0,0)[cc]{(C1)}}
\put(153,39.5){\makebox(0,0)[cc]{(C2)}}
\put(83,50.5){\makebox(0,0)[cc]
{$(X^\Gamma,\sigma_\varphi)$ has specification property}}
\put(82.75,44.25){\makebox(0,0)[cc]{(strongly stroboscopical property)}}
\end{picture}
\end{center}}
\noindent Where (Ci) denotes the counterexample $(X^{\mathbb Z},\sigma_{\varphi_i})$ for:
\begin{itemize}
\item $\varphi_1:{\mathbb Z}\to{\mathbb Z}$ with $\varphi_1(n)=n^2+1$ ($n\in{\mathbb Z}$),
\item $\varphi_2:{\mathbb Z}\to{\mathbb Z}$ with $\varphi_2(n)=-n$ ($n\in{\mathbb Z}$),
\item $\varphi_3:{\mathbb Z}\to{\mathbb Z}$ with $\varphi_3(n)=n+1$ ($n\in{\mathbb Z}$).
\end{itemize}
And:
{\scriptsize \begin{center}
\unitlength .6mm 
\linethickness{0.4pt}
\ifx\plotpoint\undefined\newsavebox{\plotpoint}\fi 
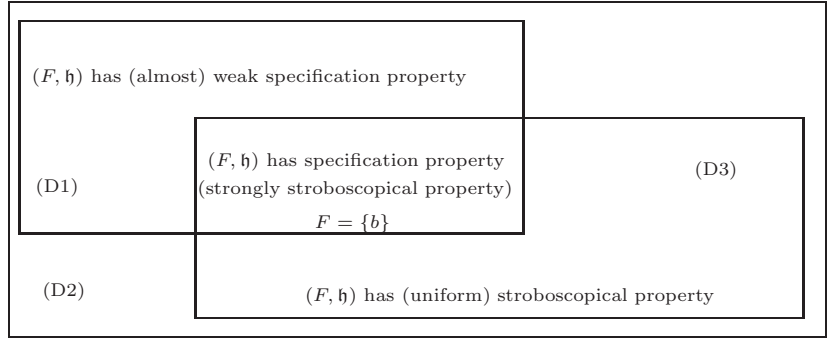
\begin{picture}(182.75,76.5)(0,0)
\put(47,59.75){\makebox(0,0)[cc]
{$\SP\SP\SP(F,\h)$ has (almost) weak specification property}}
\put(112.75,11.25){\makebox(0,0)[cc]{$(F,\h)$ has (uniform) stroboscopical property}}
\put(78.75,41.25){\makebox(0,0)[cc]{$(F,\h)$ has specification property}}
\put(78.5,35){\makebox(0,0)[cc]{(strongly stroboscopical property)}}
\put(78,27.5){\makebox(0,0)[cc]{$F=\{b\}$}}
\put(12.5,35.5){\makebox(0,0)[cc]{(D1)}}
\put(14,12.5){\makebox(0,0)[cc]{(D2)}}
\put(158.5,39.25){\makebox(0,0)[cc]{(D3)}}
\put(43.25,6.5){\framebox(134.5,44.25)[]{}}
\put(4.25,25.25){\framebox(111.5,47)[]{}}
\put(2,2.25){\framebox(180.75,74.25)[]{}}
\end{picture}
\end{center}}
\noindent Where (Di) denotes the counterexample $(\{\frac{\pm 1}{n}:n\in{\mathbb Z}\setminus\{0\}\}\cup\{0\},\h_i)$ for (consider $\{\frac{\pm 1}{n}:n\in{\mathbb Z}\setminus\{0\}\}\cup\{0\}$ with induced topology from Euclidean line $\mathbb R$):
\[\begin{array}{c}
\h_1(x)=\left\{\begin{array}{lc}  \frac{x}{x-1} &x\neq1\: , 
\\ 0 & x=1 \: , \end{array} \right.
\: \:\:\:\:\:
\h_2(x)=\left\{\begin{array}{lc} \frac{x}{x-1} & x\neq 1 \: , 
\\ -1 & x=1\: ,  \end{array} \right.
\\ \\
\h_3(x)=-x\:(x\in \{\frac{\pm 1}{n}:n\in{\mathbb Z}\setminus\{0\}\}\cup\{0\})\:.
\end{array}\]

{\scriptsize
\noindent {\bf Fatemah Ayatollah Zadeh Shirazi},  Faculty of Mthematics, Statistics and Computer Science, College of Science, University of Tehran, Enghelab Ave., Tehran, Iran (fatemah@khayam.ut.ac.ir)
\\
{\bf Zahra Nili Ahmadabadi},  Islamic Azad University, Science and Research Branch, 
Tehran, Iran 
\linebreak
(zahra.nili.a@gmail.com)
\\
{\bf Bahman Taherkhani}, Department of Mathematics, Payame Noor University, P.O.Box 19395-3697,
Tehran, Iran (bahman.taherkhani@pnu.ac.ir)
\\
{\bf Khosro Tajbakhsh}, Department of Mathematics, Faculty of Mathematical Sciences, Tarbiat Modares University,
Tehran 14115-134, Iran (khtajbakhsh@modares.ac.ir , arash@cnu.ac.kr)
}
\end{document}